\documentclass[12pt]{amsart}

\input xy
\xyoption{all}

\usepackage[utf8]{inputenc}
\usepackage{geometry}
\usepackage{amsmath}
\usepackage{amssymb}
\usepackage{amsthm}  
\usepackage{hyperref}
\usepackage{cite}

\newtheorem{same}{This should never appear}[section]
\newtheorem{defin}[same]{Definition}

\newtheorem{remark}[same]{Remark}
\newtheorem{theorem}[same]{Theorem}
\newtheorem{example}[same]{Example}
\newtheorem{lemma}[same]{Lemma}
\newtheorem{fact}[same]{Fact}
\newtheorem{question}[same]{Question}

\newtheorem{prop}[same]{Proposition}

\newtheorem{conj}[same]{Conjecture}

\newtheorem{exam}[same]{Example}

\newcommand{\fct}[2]{{}^{#1}#2}

\newbox\noforkbox \newdimen\forklinewidth
\setbox0\hbox{$\textstyle\smile$}
\setbox1\hbox to \wd0{\hfil\vrule width \forklinewidth depth-2pt
 height 10pt \hfil}
\setbox\noforkbox\hbox{\lower 2pt\box1\lower 2pt\box0\relax}
\def\unionstick{\mathop{\copy\noforkbox}\limits}

\def\nonfork_#1{\unionstick_{\textstyle #1}}

\setbox0\hbox{$\textstyle\smile$}
\setbox1\hbox to \wd0{\hfil{\sl /\/}\hfil}
\setbox2\hbox to \wd0{\hfil\vrule height 10pt depth -2pt width
              \forklinewidth\hfil}
\newbox\doesforkbox
\setbox\doesforkbox\hbox{\lower 2pt\box1 \lower 2pt\box2\lower2pt\box0\relax}
\def\nunionstick{\mathop{\copy\doesforkbox}\limits}

\def\fork_#1{\nunionstick_{\textstyle #1}}

\newcommand\Set{\operatorname{\bf Set}}
\newcommand\Met{\operatorname{\bf Met}}
\newcommand\Emb{\operatorname{\bf Emb}}

\newcommand\Lin{\operatorname{\bf Lin}}
\newcommand{\ba}{\mathbf{a}}
\newcommand{\bb}{\mathbf{b}}

\newcommand\ck{\mathcal{K}}

\newcommand{\bm}{\mathbf{m}}
\newcommand{\bx}{\mathbf{x}}
\newcommand{\by}{\mathbf{y}}

\newcommand{\ran}{\textrm{ran}}

\newcommand{\cf}{\text{cf }}
\newcommand{\rest}{\upharpoonright}

\newcommand{\lead}{\le}

\newcommand\leap[1]{\lead_{#1}}

\newcommand\lea{\leap{\K}}

\def\lee{\preceq}

\newcommand\ca{\mathcal {A}}

\newcommand\cc{\mathcal {C}}
\newcommand\cd{\mathcal {D}}

\newcommand\cl{\mathcal {L}}

\newcommand\Hom{\operatorname{Hom}}
\newcommand\Homk{\operatorname{Hom}_\ck}
\newcommand{\bigM}{\widehat{M}}
\newcommand{\bigN}{\widehat{N}}

\newcommand{\bigL}{\widehat{L}}

\newcommand{\bigK}{\widehat{\K}}

\newcommand{\K}{\operatorname{\mathcal{K}}}
\newcommand{\lan}{\operatorname{L}}
\newcommand{\LS}{\operatorname{LS}}
\newcommand{\Mod}{\operatorname{Mod}}
\newcommand{\EC}{\operatorname{EC}}
\newcommand{\PC}{\pc}
\def\l{\langle}
\def\r{\rangle}
\newcommand{\pc}{\operatorname{PC}}
\newcommand\slesseq{\trianglelefteq}
\newcommand\ksst{\lea}

\newcommand{\seq}[1]{\langle #1 \rangle}
\newcommand{\te}[1]{\textrm{#1}}
\newcommand\Mor{\operatorname{Mor}}

\newcommand{\mcard}[1]{\text{card}(#1)}
\newcommand{\Mcard}[1]{\text{card}(#1)}
\newcommand{\scard}[1]{\text{card}\left(#1\right)}

\newcommand{\gS}{\text{gS}}

\title[$\mu$-Abstract elementary classes]{$\mu$-Abstract elementary classes and other generalizations}
\date{\today}

\begin{document}

\parindent 0pt
\parskip 5pt


\author[Boney]{Will Boney}
\email{wboney@math.harvard.edu}
\urladdr{http://math.harvard.edu/\textasciitilde wboney/}
\address{Mathematics Department, Harvard University, Cambridge, Massachusetts, USA}
\thanks{This material is based upon work done while the first author was supported by the National Science Foundation under Grant No. DMS-1402191.}

\author[Grossberg]{Rami Grossberg}
\email{rami@cmu.edu}
\urladdr{http://math.cmu.edu/\textasciitilde rami/}
\address{Department of Mathematical Sciences, Carnegie Mellon University, Pittsburgh, Pennsylvania, USA}

\author[Lieberman]{Michael Lieberman}
\email{lieberman@math.muni.cz}
\urladdr{http://www.math.muni.cz/\textasciitilde lieberman/}
\address{Department of Mathematics and Statistics, Faculty of Science, Masaryk University, Brno, Czech Republic}

\author[Rosick\'y]{Ji\v r\'i Rosick\'y}
\email{rosicky@math.muni.cz}
\urladdr{http://www.math.muni.cz/\textasciitilde rosicky/}
\address{Department of Mathematics and Statistics, Faculty of Science, Masaryk University, Brno, Czech Republic}
\thanks{The third and fourth authors are supported by the Grant Agency of the Czech Republic under the grant P201/12/G028.}

\author[Vasey]{Sebastien Vasey}
\email{sebv@cmu.edu}
\urladdr{http://math.cmu.edu/\textasciitilde svasey/}
\address{Department of Mathematical Sciences, Carnegie Mellon University, Pittsburgh, Pennsylvania, USA}
\thanks{This material is based upon work done while the fifth author was supported by the Swiss National Science Foundation under Grant No.\ 155136.}

\begin{abstract}
We introduce $\mu$-Abstract Elementary Classes ($\mu$-AECs) as a broad framework for model theory that includes complete boolean algebras and metric spaces, and begin to develop their classification theory.  Moreover, we note that $\mu$-AECs correspond precisely to accessible categories in which all morphisms are monomorphisms, and begin the process of reconciling these divergent perspectives: for example, the preliminary classification-theoretic results for $\mu$-AECs transfer directly to accessible categories with monomorphisms.





\end{abstract}

\maketitle

\section{Introduction}

In this paper, we offer a broad framework for model theory, \emph{$\mu$-abstract elementary classes}, and connect them with existing frameworks, namely abstract elementary classes and, from the realm of categorical model theory, accessible categories (see \cite{makkaiparebook}, \cite{adamekrosicky}) and $\mu$-concrete abstract elementary classes (see \cite{LRmetric}).

All of the above frameworks have developed in response to the need to analyze the model theory of nonelementary classes of mathematical structures; that is, classes in which either the structures themselves or the relevant embeddings between them cannot be adequately described in (finitary) first order logic.  This project was well underway by the 50's and 60's, which saw fruitful investigations into infinitary logics and into logics with additional quantifiers (see \cite{Dic} and \cite{BaFe} for summaries).  Indeed, Shelah \cite[p.~41]{sh702} recounts that Keisler and Morley advised him in 1969 that this direction was the future of model theory and that first-order had been mostly explored.  The subsequent explosion in stability theory and its applications suggest otherwise, naturally, but the nonelementary context has nonetheless developed into an essential complement to the more classical picture.

On the model-theoretic side, Shelah was the leading figure, publishing work on excellent classes (\cite{Sh87a} and \cite{Sh87b}) and classes with expanded quantifiers \cite{l(aa)intro}, and, of greatest interest here, shifting to a formula-free context through the introduction of \emph{abstract elementary classes} (or \emph{AECs}) in \cite{Sh88}.  The latter are a purely semantic axiomatic framework for abstract model theory that encompasses first order logic as well as infinitary logics incorporating additional quantifiers and infinite conjuncts and disjuncts, not to mention certain algebraically natural examples without an obvious syntactic presentation---see \cite{nperp}.  It is important to note, though, that AECs still lack the generality to encompass the logic $L_{\omega_1, \omega_1}$ or complete metric structures.

Even these examples are captured by \emph{accessible categories}, a parallel, but significantly more general, notion developed simultaneously among category theorists, first appearing in \cite{makkaiparebook} and receiving comprehensive treatments both in \cite{makkaiparebook} and \cite{adamekrosicky}.  An accessible category is, very roughly speaking, an abstract category (hence, in particular, not a category of structures in a fixed signature) that is closed under sufficiently directed colimits, and satisfies a kind of weak Löwenheim-Skolem property: any object in the category can be obtained as a highly directed colimit of objects of small size, the latter notion being purely diagrammatic and internal to the category in question.  In particular, an accessible category may not be closed under arbitrary directed colimits, although these are almost indispensable in model-theoretic constructions: the additional assumption of closure under directed colimits was first made in \cite{rosickyjsl}---that paper also experimented with the weaker assumption of directed bounds, an idea that recurs in Section~\ref{classthysect} below.  

Subsequent work (see \cite{bekerosicky}, \cite{liebapal}, and \cite{LRclass}) has resulted in a precise characterization of AECs as concrete accessible categories with added structure, namely as pairs $(\ck,U)$, where
\begin{itemize}\item $\ck$ is an accessible category with all morphisms monomorphisms and all directed colimits, and\\
\item $U:\ck\to\Set$ (with $\Set$ the category of sets and functions) is a faithful (``underlying set'') functor satisfying certain additional axioms.\end{itemize}
Details can be found in Section 3 of \cite{LRclass}.  Of particular importance is the extent to which $U$ preserves directed colimits; that is, the extent to which directed colimits are concrete.  If we assume that $U$ preserves arbitrary directed colimits, we obtain a category equivalent to an AEC.  If we make the weaker assumption that $U$ merely preserves colimits of $\mu$-directed, rather than directed, diagrams, we arrive at the notion of a $\mu$-concrete AEC (see \cite{LRmetric}).  Note that, although directed colimits may not be preserved by $U$ (that is, they may not be ``$\Set$-like''), they still exist in the category $\ck$---metric AECs, whose structures are built over complete metric structures, are a crucial example of this phenomenon.  One might ask, though, what would happen if we weaken this still further: what can we say if we drop the assumption that $\ck$ is closed under directed colimits, and merely assume that the colimits that exist in $\ck$ and are ``$\Set$-like'' are those that are $\mu$-directed for some $\mu$?
Here we introduce a framework, called \emph{$\mu$-abstract elementary classes}, that represents a model-theoretic approximation of that generalized notion, and which, most importantly, encompasses all of the examples considered in this introduction, including classes of models in infinitary logics $L_{\kappa, \mu}$, AECs, and $\mu$-concrete AECs.  This is not done just for the sake of generalization but in order to be able to deal with specific classes of structures that allow functions with infinite arity (like $\sigma$-complete Boolean algebras or formal power series).  Moreover, such classes also occur naturally in the development of the classification theory for AECs, as can be seen by their use in \cite{indep-aec-v5} and \cite{bv-sat-v3} (there the class studied is the $\mu$-AEC of $\mu$-saturated models of an AEC).

We define $\mu$-AECs in Section \ref{prelimsection}.  We then show that the examples discussed above fit into this framework.  We establish an analog of  Shelah's Presentation Theorem for $\mu$-AECs in Section \ref{present-sec}.  In Section \ref{access-sec} we show that $\mu$-AECs are, in fact, extraordinarily general: up to equivalence of categories, the $\mu$-AECs are precisely the accessible categories whose morphisms are monomorphisms.  Although this presents certain obstacles---it follows immediately that a general $\mu$-AEC will not admit Ehrenfeucht-Mostowski constructions---there is a great deal that can be done on the $\mu$-AEC side of this equivalence. In section \ref{tamenesssect}, we show assuming the existence of large cardinals that $\mu$-AECs satisfy tameness, an important locality property in the study of AECs. In Section \ref{classthysect}, with the additional assumption of directed bounds, we begin to develop the classification theory of $\mu$-AECs. Note that the results of Sections \ref{tamenesssect} and \ref{classthysect} transfer immediately to accessible categories with monomorphisms.   

This paper was written while the fifth author was working on a Ph.D.\ thesis under the direction of the second author at Carnegie Mellon University and he would like to thank Professor Grossberg for his guidance and assistance in his research in general and in this work specifically. We also thank the referee for questions that helped us clarify some aspects of this paper.



\section{Preliminaries} \label{prelimsection}

We now introduce the notion of a \emph{$\mu$-abstract elementary class}, or $\mu$-AEC.  As with ordinary AECs, we give a semantic/axiomatic definition for a class of structures and a notion of strong substructure.

\begin{defin}\label{defmac}
Fix an infinite cardinal $\mu$.

	A \emph{$\mu$-ary language $L$} consists of a set of function symbols $\seq{F_i : i \in I_F}$ and relations $\seq{R_j : j \in J_R}$ (here, $I_F$, $J_R$ are index sets) so that each symbol has an arity, denoted $n(F_i)$ or $n(R_j)$, where $n$ is an ordinal valued function $n:\{R_j, F_i\mid i\in I_F, j\in J_R\}\rightarrow \mu$.
		
	Given a $\mu$-ary language $L$, an $L$-structure $M$ is $\seq{|M|, F^M_i, R^M_j}_{i \in L_F, j \in L_R}$ where
		 $|M|$ is a set, called the universe of $M$;
		 $F^M_i : {}^{n(F_i)}|M| \to |M|$ is a function of arity $n(F_i)$; and
		 $R^M_j \subset {}^{n(R_j)} |M|$ is a relation of arity $n(R_j)$.

	 We say that
	 $\l\K,\lea \r$ is a \emph{$\mu$-abstract class} provided
	 \begin{enumerate}
         \item $\K$ is a class of $L$-structure for a fixed $\mu$-ary language $L := \lan (\K)$.
	 \item
	 $\l\K,\lea\r$ is a partially pre-ordered class (that is, $\lea$ is reflexive and transitive) such that $M \lea N$ implies that $M$ is an $L$-submodel of $N$.
	 \item  $\l \K,\lea\r$ respects $L$-isomorphisms; that is, if $f: N \to N'$ is an $L$-isomorphism and $N \in \K$, then $N' \in \K$ and if we also have $M \in \K$ with $M \lea N$, then $f(M) \in \K$ and $f(M) \lea N'$;
	 \end{enumerate}
	
         We often do not make the distinction between $\K$ and $(\K, \lea)$.
	\end{defin}

An $L$-homomorphism is called a substructure embedding if it is injective and reflects all relations. Both inclusions of a substructure
and isomorphisms are substructure embeddings. Conversely, if $h:M\to N$ is a substructure embedding then $M$ is isomorphic to the substructure
$h(M)$ of $N$. The category of all $L$-structures and substructure embeddings is denoted by $\Emb(L)$. Then an abstract class is the same
as a subcategory $\K$ of $\Emb(L)$ which is 
\begin{enumerate}
\item
\textit{Replete}, i.e., closed under isomorphic objects.
\item \textit{Iso-full}, i.e., containing isomorphisms between $\K$-objects.
\end{enumerate}
	
Let $\l I,\leq\r$ be a partially ordered set.  We say that $I$ is \emph{$\mu$-directed}, where $\mu$ is a regular cardinal, provided that for every $J\subseteq I$ with $\scard{J}<\mu$, there exists $r\in I$ such that $r\geq s$ for all $s\in J$. Thus $\aleph_0$-directed is the usual notion of directed set. Let $\l\K,\lea\r$ be an abstract class.  A family  
$\{M_i\mid i\in I\}\subseteq\K$ is called \emph{a $\mu$-directed system} provided $I$ is a $\mu$-directed set and $i<j$ implies $M_i\lea M_j$. This
is the same as a $\mu$-directed diagram in $\K$.
	
	\begin{defin}\label{defmaec}  Suppose $\l\K,\lea\r$ is a $\mu$-abstract class, with $\mu$ a regular cardinal. We say that $\l\K, \lea \r$ is a $\mu$-\emph{abstract elementary class} if the following properties hold:
	
\begin{enumerate}

%
%

    \item \emph{(Coherence)} if $M_0, M_1, M_2 \in\K$ with $M_0  \lea  M_2$, $M_1  \lea  M_2$, and $M_0 \subseteq M_1$, then $M_0  \lea  M_1$;

    \item \emph{(Tarski-Vaught chain axioms)} If $\{M_i \in \K : i \in I\}$ is a $\mu$-directed system, then:

        \begin{enumerate}

            \item $\bigcup_{i \in I} M_i \in \K$ and, for all $j \in I$, we have $M_j  \lea  \bigcup_{i \in I} M_i$; and

            \item if there is some $N \in \K$ so that, for all $i \in I$, we have $M_i  \lea  N$, then we also have $\bigcup_{i \in I} M_i  \lea  N$.

        \end{enumerate}

    \item \emph{(L\"owenheim-Skolem-Tarski number axiom)} There exists a cardinal $\lambda = \lambda^{<\mu} \geq \scard{\lan(\K)} + \mu$ such that for any $M \in\K$ 
    and $A \subseteq |M|$, there is some $N \lea  M$ such that $A \subseteq |N|$ and $\mcard{N} \leq \scard{A}^{<\mu} + \lambda$. $\LS(\K)$ is the minimal cardinal
    $\lambda$ with this property.\footnote{Note that $\LS (\K)$ really depends on $\mu$ but $\mu$ will always be clear from context.}

\end{enumerate}

\end{defin}

Note that this definition mimics the definition of an AEC.  We highlight the key differences:
\begin{remark}\label{rmk-mu-aec} \
  \begin{enumerate}
	\item Functions and relations are permitted to have infinite arity.
	\item The L\"owenheim-Skolem-Tarski axiom only guarantees the existence of submodels of certain cardinalities, subject to favorable cardinal arithmetic.
	\item Closure under unions of $\lea$-increasing chains does not hold unconditionally: the directed systems must in fact be $\mu$-directed.
        \item The Tarski-Vaught axioms describe $\mu$-directed systems rather than chains and say that $\ck$ is closed under $\mu$-directed colimits
        in $\Emb(L)$. One could have only required that every chain of models indexed by an ordinal of cofinality at least $\mu$ has a least upper bound. When 
        $\mu = \aleph_0$, this is well-known to give an equivalent definition (see e.g. \cite{adamekrosicky} 1.7, though the central idea of the proof dates back to Iwamura's Lemma, \cite{iwamura}). In general, though, this is significantly weaker (see \cite{adamekrosicky} Exercise 1.c).  Concretely, proving the presentation theorem becomes problematic if one opts instead for the chain definition.
        \item\label{rmk-lambda-ary} Replacing the Tarski-Vaught axioms by $\K$ being closed under $\mu$-directed colimits in $\Emb(L)$ makes sense also when $\K$ is a $\lambda$-abstract class for $\lambda>\mu$.
        \item As a notational remark, we use $\scard{}$ to denote the size of sets and (universes) of models.  This breaks with convention, but is to avoid $|\cdot|$ being used to denote both universe and cardinality, leading to the notation $\|M\|$ for the cardinality of the universe of a model.
\end{enumerate}
\end{remark}

As promised in the introduction, $\mu$-AECs subsume many previously studied model theoretic frameworks:
\begin{enumerate}
  
	\item All AECs are $\aleph_0$-AECs with the same L\"{o}wenheim-Skolem number.  This follows directly from the definition.  See \cite{grossberg2002} for examples of classes of structures that are AECs; in particular, AECs subsume classical first-order model theory.
	
	\item Given an AEC $\K$ with amalgamation (such as models of a first order theory), and a regular cardinal $\mu > \LS (\K)$, the class of $\mu$-saturated\footnote{In the sense of Galois types.} models of $\K$ is a $\mu$-AEC with Löwenheim-Skolem number $\LS (\K)^{<\mu}$.  If $\K$ is also tame and stable (or superstable), then the results of Boney and Vasey \cite{bv-sat-v3} show that this is true even for certain cardinals below the saturation cardinal $\mu$.
	
	\item \label{inflog-ex} Let $\lambda \geq \mu$ be cardinals with $\mu$ regular. Let $L_A$ be a fragment of $L_{\lambda, \mu}$ (recall that a fragment is a collection of formulas closed under sub formulas and first order connectives), and let $T$ be a theory in that fragment.  Then $\K= (\Mod \te{ } T, \preceq_{L_A})$ is a $\mu$-AEC with $\LS(\K) = (\scard{L_A} + \scard{T})^{<\mu}$, where $M \preceq_{L_A} N$ if and only if for all $\phi(\bx) \in L_A$ and $\bm \in |M|$ of matching arity (which might be infinite), we have that $M \models \phi[\ba]$ if and only if $N \models \phi[\ba]$.
	
	\item Complete metric spaces form an $\aleph_1$-AECs.  This follows from the above item because metric spaces are axiomatizable in first order and completeness is axiomatized by the $L_{\omega_1, \omega_1}$ sentence
	$$\forall \seq{x_n : n < \omega} [ (\wedge_{\epsilon \in \mathbb{Q}^+} \vee_{N < \omega} \wedge_{N <n < m < \omega} d(x_n, x_m) < \epsilon) \implies \exists y (\wedge_{\epsilon \in \mathbb{Q}^+} \vee_{N < \omega} \wedge_{N <n  < \omega} d(x_n, y) < \epsilon)]$$
	Although this does not capture the $[0,1]$-value nature of many treatments of the model theory of metric structures, such as \cite{BBHU}, this can be incorporated in one of two ways.  One could add the real numbers as a second sort, interpret relations as functions between the sorts, and axiomatize all of the continuity properties.  A less direct approach is taken in \cite{continf}, where a complete structure is approximated by a dense subset describable in $L_{\omega_1, \omega}$.

\item Along the lines of complete metric spaces, $\mu$-complete boolean algebras are $\mu$-AECs because $\mu$-completeness can be written as a $L_{\mu, \mu}$-sentence.

\item \label{caec-ex} Any $\mu$-concrete AEC (or $\mu$-CAEC), in the sense of \cite{LRmetric}, is a $\mu$-AEC.
      \item Any $\mu$-ary functorial expansion of a $\mu$-AEC is naturally a $\mu$-AEC. See Section 2.1 immediately below.
      
      \item Generalizing $L(Q)$, consider classes axiomatized by $L_{\lambda, \mu}(Q^\chi)$, where $Q^\chi$ is the quantifier ``there exist at least $\chi$'' (the standard $L(Q)$ is $L_{\omega, \omega}(Q^{\aleph_1})$ in this notation).  As in (\ref{inflog-ex}), let $T$ be a theory in $L_{\lambda, \mu}(Q^\chi)$ and $L_A$ be a fragment of this logic containing $T$.  Since $Q^\chi$ is $L_{\chi, \chi}$ expressible, we already have $\K_0 := (\Mod T, \leap{L_A})$ is a $(\mu+\chi)$-AEC with $\LS(\K_0) = (\scard{L_A}+\chi)^{<(\mu + \chi)}$.  For a stronger result, if we set $M \leap{L_A}^* N$ by
      \begin{center} 
      $M \leap{L_A} N$ and if $Q^\chi x \phi(x, \by) \in L_A$ with $M\vDash \neg Q^\chi x \phi(x, \bm)$, then $\phi(M, \ba) = \phi(N, \ba)$
      \end{center}
      then $\K_1 := (\Mod T, \leap{L_A}^*)$ is a $\mu$-AEC with $\LS(\K_1) = (\scard{L_A}+\chi)^{<\mu}$.  Moreover, if $\chi = \chi_0^+$ and $L_A$ only contains \emph{negative} instances of $Q^\chi$, then $\LS(\K_1) = (\scard{L_A} + \chi_0)^{<\mu}$.

\end{enumerate}

We now briefly discuss the interplay between certain $\mu$-AECs and functorial expansions.

\subsection{Functorial expansions and infinite summation}

Recall from \cite[Definition 3.1]{sv-infinitary-stability-v5}:

\begin{defin}
  Let $\K$ be a $\mu$-AEC with $L = \lan(\K)$ and let $\bigL$ be a $\lambda$-ary expansion of $L$ with $\lambda \ge \mu$. A \emph{$\lambda$-ary $\bigL$-functorial expansion of $\K$} is a class $\bigK$ of $\bigL$-structures satisfying:
        \begin{enumerate}
          \item The map $\bigM \mapsto \bigM \rest L$ is a bijection from $\bigK$ onto $\K$. For $M \in\K$, we write $\bigM$ for the unique element of $\bigK$ whose reduct is $M$.
          \item Invariance: If $f: M \cong N$, then $f: \bigM \cong \bigN$.
          \item Monotonicity: If $M \lea N$, then $\bigM \subseteq \bigN$.
        \end{enumerate}
        
        We order $\bigK$ by $\bigM \leap{\bigK} \bigN$ if and only if $M \lea N$.
\end{defin}

\begin{fact}[Proposition 3.8.(4) in \cite{sv-infinitary-stability-v5}]
  Let $\K$ be a $\mu$-AEC and let $\bigK$ be a $\mu$-ary functorial expansion of $\K$. Then $(\bigK, \leap{\bigK})$ is a $\mu$-AEC with $\LS (\bigK) = \LS (\K)$.
\end{fact}

\begin{remark}
  A word of warning: if $\K$ is an AEC and $\bigK$ is a functorial expansion of $\K$, then $\K$ and $\bigK$ are isomorphic (as categories). In particular, any directed system in $\bigK$ has a colimit. However, $\bigK$ may \emph{not} be an AEC if $\lan (\bigK)$ is not finitary: the colimit of a directed system in $\bigK$ may \emph{not} be the union: relations may need to contain more elements. However, if we change the definition of AEC to allow languages of infinite arity (see Remark \ref{rmk-mu-aec}.(\ref{rmk-lambda-ary})), then $\bigK$ will be an AEC in that new sense, i.e.\ an ``infinitary'' AEC.
\end{remark}

\begin{remark}
Let $\K$ be a $\mu$-AEC and consider a $\bigL$-functorial expansion $\bigK$ of $\K$. Then any function and relation symbols from 
$\bigL$ are \textit{interpretable} in $\K$ in the sense of \cite{concr} (this idea goes back to \cite{law}). This means that function symbols of arity $\alpha$ are natural transformations $\varphi:U^\alpha\to U$ where $U:\K\to\Set$ is the forgetful functor (given as the domain restriction of the forgetful functor $\Emb(L)\to\Set$ assigning underlying sets to $L$-structures) and $U^\alpha$ is the functor $\Set(\alpha,U(-)):\K\to\Set$. Similarly, relation symbols of arity $\alpha$ are subfunctors $R$ of $U^\alpha$. 

If $\bigL$ is $\mu$-ary then subfunctors $R$ preserve $\mu$-directed colimits. Since $\K$ is an $\LS(\K)^+$-accessible category (see \ref{kaecsacc}), both
$\varphi$ and $R$ are determined by their restrictions to the full subcategory $\K_{\LS(\K)^+}$ of $\K$ consisting of $\LS(\K)^+$-presentable objects.
Since there is only a set of such objects, there is a largest $\mu$-ary functorial expansion where $\bigL$ consists of all symbols for natural transformations and subfunctors as above. For 
$\mu=\aleph_0$, this is contained in \cite{LRclass} Remark 3.5.
\end{remark}

The main example in \cite{sv-infinitary-stability-v5} is Galois Morleyization (Definition 3.3 there). However there are many other examples including the original motivation for defining $\mu$-AECs: infinite sums in boolean algebras. The point is that even though the language of boolean algebras with a sum operator is infinitary, we really need only to work in an appropriate class in a finitary language that we functorially expand as needed. This shows in a precise sense that the infinite sum operator is already implicit in the (finitary) structure of boolean algebras themselves.

\begin{defin}
  Fix infinite cardinals $\lambda \ge \mu$. Let $\Phi$ be a set of formulas in $L_{\lambda, \mu}$. Let $(\K, \lea)$ be an abstract class. Define $\K_\Phi := (\K, \leap{\K_\Phi})$ by $M \leap{\K_\Phi} N$ if and only if $M \lea N$ and $M \lee_{\Phi} N$.
\end{defin}

\begin{lemma}
  Let $\lambda \ge \mu$, $\Phi$ be a set of formulas in $L_{\lambda , \mu}$. Let $(\K, \lea)$ be a $\mu_0$-AEC with $\mu_0 \le \mu$. Then:

  \begin{enumerate}
    \item $K_{\Phi}$ is a $\mu$-AEC.
    \item If all the formulas in $\Phi$ have fewer than $\mu_0$-many quantifiers, then $\K_{\Phi}$ satisfies the first Tarski-Vaught chain axiom of $\mu_0$-AECs.
  \end{enumerate}
\end{lemma}
\begin{proof}
  The first part is straightforward. The second is proven by induction on the quantifier-depth of the formulas in $\Phi$.
\end{proof}

\begin{example}
Let $T$ be a completion of the first-order theory of boolean algebras and let $\K := (\Mod (T), \lee)$. Let $\Phi$ consist of the $L_{\omega_1, \omega_1}$ formula $\phi (\bx, y)$ saying that $y$ is a least upper bound of $\bx$ (here $\ell (\bx) = \omega$). Then $\phi$ has only one universal quantifier so by the Lemma, $\K_{\Phi}$ satisfies the first Tarski-Vaught chain axiom of AECs. Of course, $\K_{\Phi}$ is also an $\aleph_1$-AEC. Now expand each $M \in \K$ to $\bigM$ by defining $R_{\Sigma}^M (\ba, b)$ to hold if and only if $b$ is a least upper bound of $\ba$ (with $\ell (\ba) = \omega$). Let $\bigK_{\Phi} := \{\bigM \mid M \in \K\}$. Then one can check that $\bigK_{\Phi}$ is a functorial expansion of $\K_{\Phi}$.
\end{example}

Basic definitions and concepts for AECs, such as amalgamation or Galois types (see \cite{baldwinbook} or \cite{ramibook} for details), can be easily transferred to $\mu$-AECs. In the following sections, we begin the process of translating essential theorems from AECs to $\mu$-AECs.

\section{Presentation Theorem} \label{present-sec}

We now turn to the Presentation Theorem for $\mu$-AECs.  This theorem has its conceptual roots in Chang's Presentation Theorem \cite{Ch}, which shows that $L_{\lambda, \omega}$ can be captured in a larger finitary language by omitting a set of types.  A more immediate predecessor is Shelah's Presentation Theorem, which reaches the same conclusion for an arbitrary AEC.  Unfortunately, while Chang's Presentation Theorem gives some insight into the original class, Shelah's theorem does not.  However, the presentation is still a useful tool for some arguments and provides a syntactic characterization of what are otherwise purely semantic objects.

\begin{defin}
Let $L \subset L_1$ be $\mu$-ary languages, $T_1$ an $(L_1)_{\mu, \mu}$ theory, and $\Gamma$ be a set of $\mu$-ary $(L_1)_{\mu, \mu}$-types.  Here we define a \emph{$\mu$-ary $(L_1)_{\mu, \mu}$-type} as a set of $(L_1)_{\mu, \mu}$ formulas in the same free variables $\bx$, where $\bx$ has arity less than $\mu$.  We define  
\begin{itemize}
	\item $\EC^\mu(T_1, \Gamma) = \{ M : M \te{an } L_1\te{-structure,}\,\,M\models T_1, M\te{omits each type in }\Gamma \}$
	\item $\PC^\mu(T_1, \Gamma, L) = \{ M \rest L : M \in \EC^\mu(T_1, \Gamma) \}$
\end{itemize}
\end{defin}

\begin{theorem}\label{pres-thm}
Let $\K$ be a $\mu$-AEC with $\LS(\K) = \chi$.  Then we can find some $L_1 \supset \lan(\K)$, a $(L_1)_{\mu, \mu}$-theory $T_1$ of size $\chi$, and a set $\Gamma$ of $\mu$-ary $(L_1)_{\mu, \mu}$-types with $\scard{\Gamma} \le 2^\chi$ so that $\K = \PC^\mu(T_1, \Gamma, \lan(\K))$.
\end{theorem}

Although we don't state them here, the traditional moreover clauses (see e.g.\ the statement of \cite[Theorem 4.15]{baldwinbook}) apply as well.

\begin{proof} 
We adapt the standard proofs; see, for instance, \cite[Theorem 4.15]{baldwinbook}.  Set $\chi := \LS(\K)$.  We introduce ``Skolem functions'' $L_1 := L(\K) \cup \{ F_i^\alpha(\bx) : i < \chi, \ell(\bx) = \alpha < \mu\}$ and make very minimal demands by setting
$$T_1 := \{ \exists x(x = x) \} \cup \{ \forall \bx F_i^\alpha(\bx) = x_i : i <\alpha < \mu\}$$
For any $M_1 \vDash T_1$ and $\ba \in |M_1|$, we define $N_\ba^{M_1}$ to be the minimal $L_1$-substructure of $M_1$ that contains $\ba$.  We can code the information about $N_\ba^{M_1}$ into $\ba$'s quantifier-free type:
$$p_\ba^{M_1} := \{ \phi(\bx) : \phi(\bx) \in (L_1)_{\mu, \mu} \text{ is quantifier-free and } M_1 \vDash \phi(\ba) \}$$
Given tuples $\ba \in M_1$ and $\bb \in M_2$ of the same length, we have that $p_\ba^{M_1} = p_\bb^{M_2}$ if and only if the map taking $\ba$ to $\bb$ induces an isomorphism $N_\ba^{M_1} \cong N_\bb^{M_2}$.  Since we have this tight connection between types and structures, we precisely want to exclude types that give rise to structures not coming from $\K$.  Thus, we set
\begin{eqnarray*}
\Gamma^{M_1} &:=& \bigcup \{ p_\ba^{M_1} : \exists \bb \subset \ba \text{ such that } (N^{M_1}_\bb) \rest L(\K) \not \lea  (N^{M_1}_\ba) \rest L(\K)\}\\
\Gamma &:=& \bigcup_{M_1 \vDash T_1} \Gamma^{M_1}
\end{eqnarray*}
Note that the $\bb$ in the first line might be $\ba$, in which case the condition becomes $N_\ba^{M_1} \rest L(\K) \not\in \K$.  By counting the number of $(L_1)_{\mu, \mu}$-types, we have that $\scard{\Gamma} \le 2^\chi$. Now all we have left to show is the following claim.

{\bf Claim:} $\K = \PC(T_1, \Gamma, \lan(\K))$\\
First, let $M_1 \in \EC(T_1, \Gamma)$.  Given $\ba \in {}^{<\mu}|M_1|$, we know that $\ba \vDash p_\ba^{M_1}$ so $p_\ba^{M_1} \notin \Gamma$.  Thus, $\{ N_\ba^{M_1} \rest \lan(\K) : \ba \in {}^{< \mu} |M_1| \}$ is a $\mu$-directed system from $\K_{\le \chi}$ with union $M_1 \rest \lan(\K)$, so $M_1 \rest \lan(\K) \in\K$.

Second, let $M \in\K$.  We need to define an expansion $M_1 \in \EC(T_1, \Gamma)$.  We can build a directed system $\{M_\ba \in\K_{\chi} : \ba \in {}^{< \mu}|M|\}$.  Since each $M_\ba$ has size $\chi$, we can define the $F_i^\alpha$ by enumerating $|M_\ba| = \{ F^{\ell(\ba)}_i(\ba) : i < \chi\}$ with the condition that $F^{\ell(\ba)}_i(\ba) = a_i$ for $i < \ell(\ba)$.  This precisely defines the expansion $M_1 := \seq{M, F_i^\alpha}_{i<\chi, \alpha<\mu}$.  It is easy to see $M_1 \vDash T_1$.  We also have $N_\ba^{M_1} \rest L(\K) = M_\ba$, so $M_1$ omits $\Gamma$ because $\{M_\ba : \ba \in {}^{<\mu}|M|\}$ is a $\mu$-directed system from $\K$.  So $M \in PC^\mu(T_1, \Gamma, L)$.
\end{proof}

\begin{remark}
  A consequence of the presentation theorem for AECs is that an AEC $\K$ with a model of size $\beth_{(2^{\LS (\K)})^+}$ has arbitrarily large models (see e.g.\ \cite[Corollary 4.26]{baldwinbook}). The lack of Hanf numbers for $L_{\mu, \mu}$ means that we cannot use this to get similar results for $\mu$-AECs.  Thus the following question is still open: Can we compute a bound for the Hanf number $H(\lambda, \mu)$, where any $\mu$-AEC $\K$ with $\LS(\K) \leq \lambda$ that has a model larger than $H(\lambda, \mu)$ has arbitrarily large models?  
\end{remark}

\section{$\mu$-AECS and accessible categories} \label{access-sec}

Accessible categories were introduced in \cite{makkaiparebook} as categories closely connected with categories of models of $L_{\kappa,\lambda}$
theories. Roughly speaking, an accessible category is one that is closed under certain directed colimits, and whose objects can be built via certain directed colimits of a set of small objects.  To be precise, we say that a category $\ck$ is $\lambda$-\textit{accessible}, $\lambda$ a regular cardinal, if it closed under $\lambda$-directed colimits (i.e. colimits indexed by a $\lambda$-directed poset) and contains, up to isomorphism, a set $\ca$ 
of $\lambda$-presentable objects such that each object of $\ck$ is a $\lambda$-directed colimit of objects from $\ca$. Here $\lambda$-presentability functions as a notion of size that makes sense in a general, i.e. non-concrete, category: we say an object $M$ is $\lambda$-\textit{presentable} if its hom-functor $\ck(M,-):\ck\to\Set$ preserves $\lambda$-directed colimits. Put another way, $M$ is $\lambda$-presentable if for any morphism $f:M\to N$ with $N$ a $\lambda$-directed colimit $\langle \phi_\alpha:N_\alpha\to N\rangle$, $f$ factors essentially uniquely through one of the $N_\alpha$, i.e. $f=\phi_\alpha f_\alpha$ for some $f_\alpha:M\to N_\alpha$. 

For each regular cardinal $\kappa$, an accessible category $\ck$ contains, up to isomorphism, only a set of $\kappa$-presentable objects.
Any object $M$ of a $\lambda$-accessible category is $\kappa$-presentable for some regular cardinal $\kappa$. Given an object $M$, the smallest cardinal $\kappa$ such that $M$ is $\kappa$-presentable is called the \textit{presentability rank} of $M$. If the presentability rank of $M$ is a successor cardinal 
$\kappa=\|M\|^+$ then $\|M\|$ is called the \textit{internal size} of $M$ (this always happens if $\ck$ has directed colimits or under GCH, see \cite{bekerosicky} 4.2 or 2.3.5).  This notion of size internal to a particular category more closely resembles a notion of dimension---in the category $\Met$ of complete metric spaces with isometric embeddings, for example, the internal size of an object is precisely its density character---and, even in case the category is concrete, may not correspond to the cardinality of underlying sets.  This distinction will resurface most clearly in the discussion at the beginning of Section~\ref{classthysect} below.

We consider the category-theoretic structure of $\mu$-AECs.  As we will see, for any uncountable cardinal $\mu$, any $\mu$-AEC with L\"owenheim-Skolem-Tarski number $\lambda$ is a $\lambda^+$-accessible category whose morphisms are monomorphisms, and that (perhaps more surprisingly) any $\mu$-accessible category whose morphisms are monomorphisms is equivalent to a $\mu$-AEC with L\"owenheim-Skolem-Tarski number $\lambda=\max(\mu,\nu)^{<\mu}$, where $\nu$, discussed in detail below, is the number of morphisms between $\mu$-presentable objects.  

It is of no small interest that a general $\mu$-accessible category also satisfies a L\"owenheim-Skolem-Tarski axiom of sorts, governed by the sharp inequality relation, $\slesseq$\footnote{The sharp inequality was introduced by Makkai and Pare \cite[Section 2.3]{makkaiparebook} and is defined by $\kappa \slesseq \kappa'$ if and only if every $\kappa$-accessible category is also a $\kappa'$-accessible category, among other equivalent conditions.}.  As we will see, this notion (see \cite{makkaiparebook}) matches up perfectly with the behavior of $\mu$-AECs conditioned by axiom \ref{defmaec}(3).

We wish to show that $\mu$-AECs and accessible categories are equivalent.  For the easy direction---that every $\mu$-AEC is accessible---we simply follow the argument for the corresponding fact for AECs in Section 4 of \cite{liebapal}.

\begin{lemma} Let $\ck$ be a $\mu$-AEC with L\"owenheim-Skolem-Tarski number $\lambda$.  Any $M\in\ck$ can be expressed as a $\lambda^+$-directed union of its $\ksst$-substructures of size at most $\lambda$.\end{lemma}
\begin{proof}  Consider the diagram consisting of all $\ksst$-substructures of $M$ of size at most $\lambda$ and with arrows the $\ksst$-inclusions.  To check that this diagram is $\lambda^+$-directed, we must show that any collection of fewer than $\lambda^+$ many such submodels have a common extension also belonging to the diagram.  Let $\{M_\alpha\,|\,\alpha<\nu\}$, $\nu<\lambda^+$, be such a collection.  Since $\lambda^+$ is regular, $\sup\{|M_\alpha|\,|\,\alpha<\nu\}<\lambda^+$, whence
$$\Mcard{\bigcup_{\alpha<\nu}M_\alpha}\leq\nu\cdot\sup\{\mcard{M_\alpha}\,|\,\alpha<\nu\}\leq\nu\cdot\lambda=\lambda$$
This set will be contained in some $M'\lea M$ with $\mcard{M'}\leq \lambda^{<\mu}+\lambda=\lambda+\lambda=\lambda$, by the L\"owenheim Skolem-Tarski axiom.  For each $\alpha<\nu$, $M_\alpha\ksst M$ and $M_\alpha\subseteq M'$.  Since $M'\ksst M$, coherence implies that $M_\alpha\ksst M'$.  So we are done. \end{proof}

\begin{lemma}\label{aecpres} 
Let $\ck$ be a $\mu$-AEC with L\"owenheim-Skolem-Tarski number $\lambda$.  A model $M$ is $\lambda^+$-presentable in $\ck$ if and only if $\mcard{M}\leq\lambda$.
\end{lemma}
\begin{proof} See the proof of Lemma 4.3 in \cite{liebapal}.\end{proof}

Taken together, these  lemmas imply that any $\mu$-AEC with L\"owenheim-Skolem-Tarski number $\lambda$ contains a set of $\lambda^+$-presentable objects, namely $\ck_{<\lambda^+}$, and that any model can be built as a $\lambda^+$-directed colimit of such objects.  As the Tarski-Vaught axioms ensure closure under $\mu$-directed colimits and $\lambda\geq\mu$, it follows that $\ck$ is closed under $\lambda^+$-directed colimits.  Thus we have:

\begin{theorem}\label{kaecsacc} Let $\ck$ be a $\mu$-AEC with L\"owenheim-Skolem-Tarski number $\lambda$.  Then $\ck$ is a $\lambda^+$-accessible category.
\end{theorem}

\begin{remark}\label{down}
Theorem \ref{kaecsacc} is valid for any $\lambda$ from \ref{defmaec}(3) and not only for the minimal one. Moreover, we only need that $\lambda$ satisfies
the  L\"owenheim-Skolem-Tarski property for $\scard{A}\leq\lambda$. In this case, we will say that $\lambda$ is a \emph{weak L\"owenheim-Skolem-Tarski number}.
\end{remark}

We now aim to prove that any accessible category whose morphisms are monomorphisms is a $\mu$-AEC for some $\mu$. In fact, there are two cases delineated below, concrete and abstract.  In Theorem~\ref{niceacckaec} we consider the concrete case: $\ck$ is taken to be a $\kappa$-accessible category of $L$-structures and 
$L$-embeddings for some $\mu$-ary language $L$ where $\mu+\scard{L}\leq\kappa$. In particular, we insist that $\ck$ sits nicely in $\Emb(L)$, the category of all $L$-structures and substructure embeddings---we may assume $L$ is relational.  In Theorem~\ref{equivthm}, we consider abstract accessible categories, with no prescribed signature or underlying sets. 

\begin{theorem}\label{niceacckaec} Let $L$ be a $\mu$-ary signature and $\ck$ be an iso-full, replete and coherent $\kappa$-accessible subcategory
of $\Emb(L)$ where $\mu+\scard{L}\leq\kappa$. If $\ck$ is closed under $\mu$-directed colimits in $\Emb(L)$ and the embedding $\ck\to\Emb(L)$ preserves 
$\kappa$-presentable objects then $\ck$ is a $\mu$-AEC with $\LS(\K)\leq\lambda=\kappa^{<\mu}$.\end{theorem}


\begin{proof}  We verify that $\ck$ satisfies the axioms in Definitions~\ref{defmac} and \ref{defmaec}.

Given such a category, we define the relation $\ksst$ as we must: for $M, N\in\ck$, $M\ksst N$ if and only if $M\subseteq N$ and the inclusion is a morphism in $\ck$.  Axiom~\ref{defmac}(1) follows immediately from this definition.  Axiom~\ref{defmac}(2) follows from the assumption that the inclusion $E$ is replete and iso-full, while~\ref{defmaec}(1) follows from the assumption that the aforementioned inclusion is a coherent functor.  \ref{defmaec}(2) is easily verified: given a $\mu$-directed system $\langle M_i\,|\,i\in I\rangle$ in $\K$, the colimit lies in $\ck$ (by $\mu$-accessibility), and since the inclusion $E$ preserves $\mu$-directed colimits, it will be precisely the union of the system.  So $\ck$ is closed under $\mu$-directed unions.  The other clauses of~\ref{defmaec}(2) are clear as well.

Axiom~\ref{defmaec}(3), the L\"owenheim-Skolem-Tarski Property, poses more of a challenge.  To begin, we recall that in $\Emb(L)$, an object is $\kappa^+$-presentable for $\kappa=\kappa^{<\mu}\geq\mu+\scard{L}$ precisely if its underlying set is of cardinality at most $\kappa$.  

Recall that we intend to show that $\lambda=\kappa^{<\mu}$ satisfies \ref{defmaec}(3).  Let $M\in\ck$ and $A\subseteq |M|$ with $|A|=\alpha>\lambda$. We begin by showing that $\ck$ is $(\alpha^{<\mu})^+$-accessible. This is an consequence of \cite{LRmetric} 4.10 because $\kappa\leq\lambda<(\alpha^{<\mu})^+$ and 
$\mu\slesseq(\alpha^{<\mu})^+$. The sharp inequality is a consequence of Example 2.13(4) in \cite{adamekrosicky}: for any cardinals $\beta<(\alpha^{<\mu})^+$ 
and $\gamma<\mu$,
$$\beta^{\gamma}\leq(\alpha^{<\mu})^\gamma=\alpha^{<\mu}<(\alpha^{<\mu})^+.$$
Since $\ck$ is $(\alpha^{<\mu})^+$-accessible, we can express $M$ as an $(\alpha^{<\mu})^+$-directed colimit of $(\alpha^{<\mu})^+$-presentable objects in 
$\ck$, say $\langle M_i\to M\,|\,i\in I\rangle$---indeed, we may assume without loss that this is a $(\alpha^{<\mu})^+$-directed system of inclusions.  
Following \cite{LRmetric} 4.6, $E:\ck\to\Emb(L)$ preserves $(\alpha^{<\mu})^+$-presentable objects---hence the $M_i$ are also $(\alpha^{<\mu})^+$-presentable 
in $\Emb(L)$, and thus of cardinality at most $\alpha^{<\mu}=\scard{A}^{<\mu}$, by the remark in the previous paragraph.  For each $a\in A$, choose $M_{i_a}$ with 
$a\in |M_{i_a}|$.  The set of all such $M_{i_a}$ is of size at most $\alpha<(\alpha^{<\mu})^+$ and we have chosen the colimit to be $(\alpha^{<\mu})^+$-directed, so there is some $M'=M_j$, $j\in I$, with $M_{i_a}\ksst M'$ for all $a\in A$.  Hence $A\subseteq |M'|$, $M'\ksst M$, and 
$$\mcard{M'}\leq \alpha^{<\mu}\leq\alpha^{<\mu}+\lambda=\scard{A}^{<\mu}+\lambda.$$
We now consider the case $\scard{A}\leq\lambda$. Hence 
$$\scard{A}^{<\mu}\leq \lambda^{<\mu} = \lambda$$
So the cardinal bound in the L\"owenheim-Skolem-Tarski Property defaults to $\lambda$. Since $\mu\slesseq\lambda^+$ (by \cite{adamekrosicky} 2.13(4) again)
and $\kappa\leq\lambda^+$, \cite{LRmetric} 4.10 and 4.6 imply that $\ck$ is $\lambda^+$-accessible and the functor $E:\ck\to\Emb(L)$ preserves $\lambda^+$-presentable objects.
Thus we may use the same argument as above to find $M'\ksst M$ of size $\lambda$ containing $A$.
\end{proof}

\begin{remark}\label{down2}
Following \ref{down} and \ref{niceacckaec}, any $\mu$-abstract class from \ref{defmaec} with (3) weakened to the existence of a weak 
L\"owenheim-Skolem-Tarski number $\lambda$
is a $\mu$-AEC with $\LS(\ck)\leq(\lambda^+)^{<\mu}$.

Assuming Vop\v enka's principle, the weak L\"owenheim-Skolem-Tarski number axiom is satisfied by any full subcategory $\ck$ of $\Emb(L)$. This follows from \cite{preacc}
and is related to the unpublished theorem of Stavi (see \cite{magidor}).  
\end{remark}

To summarize, we have so far shown that any reasonably embedded $\kappa$-accessible subcategory of a category of structures $\Emb(L)$ is a $\mu$-AEC. We wish 
to go further, however: given any $\mu$-accessible category whose morphisms are monomorphisms, we claim that it is equivalent---as an abstract category---to 
a $\mu$-AEC, in a sense that we now recall.  

\begin{defin}We say that categories $\cc$ and $\cd$ are \emph{equivalent} if the following equivalent conditions (see \cite{cwm} V.4.1) hold:
\begin{enumerate}\item There is a functor $F:\cc\to\cd$ that is 
\begin{itemize}\item \emph{full}: For any $C_1,C_2$ in $\cc$, the map $f\mapsto F(f)$ is a surjection from $\Hom_{\cc}(C_1,C_2)$ to $\Hom_{\cd}(FC_1,FC_2)$.
\item \emph{faithful}: For any $C_1,C_2$ in $\cc$, the map $f\mapsto F(f)$ is an injection from $\Hom_{\cc}(C_1,C_2)$ to $\Hom_{\cd}(FC_1,FC_2)$. 
\item \emph{essentially surjective}: Any object $D$ in $\cd$ is isomorphic to $F(C)$ for some $C$ in $\cc$.\end{itemize}
\item There are functors $F:\cc\to\cd$ and $G:\cd\to\cc$ such that the compositions $FG$ and $GF$ are naturally isomorphic to the identity functors on $\cd$ and $\cc$, respectively.\end{enumerate}\end{defin}

One might insist that the compositions in condition (2) are in fact equal to the identity functors, but this notion (\emph{isomorphism of categories}) is typically too strong to be of interest---equivalence of categories as described above is sufficient to ensure that a pair of categories exhibit precisely the same properties.  In particular, if $F:\cc\to\cd$ gives an equivalence of categories, it preserves and reflects internal sizes and gives a bijection between the isomorphism classes in $\cc$ and those in $\cd$; thus questions of, e.g., categoricity have identical answers in $\cc$ and $\cd$.

We proceed by constructing, for a general $\mu$-accessible category $\ck$ whose morphisms are monomorphisms, a full, faithful, essentially surjective functor from $\ck$ to $\ck'$, where $\ck'$ is a $\mu$-AEC.  We begin by realizing a $\mu$-accessible category as a category of structures.

\begin{lemma}Let $\ck$ be a $\mu$-accessible category whose morphisms are monomorphisms. There is a unary many-sorted signature $L$ such that $\ck$ is fully embedded to an equational variety in $\Emb(L)$.

Moreover, this full embedding preserves $\mu$-directed colimits.\end{lemma}
\begin{proof} Let $\ca$ be the full subcategory of $\mu$-presentable objects in $\ck$ (technically, we want $\ca$ to be skeletal, which makes it small).  Consider the canonical embedding 
$$E:\ck\to\Set^{\ca^{op}}$$
that takes each $K\in\ck$ to the contravariant functor $\Homk(-,K)\!\!\upharpoonright_{\ca^{op}}$, and each $\ck$-morphism $f:K\to K'$ to the natural transformation $E(f):\Homk(-,K)\to\Homk(-,K)$ given by postcomposition with $f$.  We note that, by Proposition~2.8 in \cite{adamekrosicky}, this functor is fully faithful and preserves $\mu$-directed colimits.  In fact, we may identify the image $\cl_1$ of $\ck$ in $\Set^{\ca^{op}}$ with an equational variety.  Let $L$ be a signature with sorts $\{S_A\,|\,A\in\ca\}$, and with unary function symbols for each morphism in $\ca$, i.e. a function symbol $\bar{f}$ of arity $S_A\to S_B$ for each $\ca$-map $f:B\to A$, subject to certain equations: whenever $h=f\circ g$ in $\ca$, we insist that $\bar{h}=\bar{g}\bar{f}$.  Concretely, the identification is given by a functor $F:\cl_1\to\Emb(L)$ that takes each functor $E(K)=\Homk(-,K)$ to the structure $FE(K)$ with sorts $S_A^{FE(K)}=\Homk(A,K)$ and with each $\bar{f}:S_A\to S_B$ interpreted as the function $\bar{f}^{FE(K)}:\Homk(A,K)\to\Homk(B,K)$ given by precomposition with $f$.  Any morphism $g:K\to K'$ in $\ck$ is first sent to the natural transformation $E(g):\Homk(-,K)\to\Homk(-,K')$ then sent, via $F$, to $FE(g):FE(K)\to FE(K')$, which is given sortwise by postcomposition with $g$, i.e. for any $A\in\ca$ and $f\in S_A^{FE(K)}=\Homk(A,K)$, $g(f)=g\circ f$.  Clearly, morphisms are injective in the image of $\ck$ under $FE$, as they come from monomorphisms in $\ck$, and they trivially reflect relations, meaning that in fact $F:\cl_1\to\Emb(L)$.
\end{proof}

Let $\cl_2$ denote the image of $\ck$ in $\Emb(L)$ under $FE$.  So we have exhibited $\ck$ as a full subcategory of $\Emb(L)$ closed under $\mu$-directed colimits, where $\Sigma$ is a finitary language.  As a result, the induced relation $\ksst$ is simply $\subseteq$, and iso-fullness and repleteness of the embedding are trivial.  There is only one more wrinkle that we need to consider: the presentability rank of structures in the image of $\ck$ in $\Emb(L)$ need not correspond to the cardinality of the union of their sorts---that is, if $U$ denotes the forgetful functor $\Emb(L)\to\Set$, a $\mu$-presentable object $FE(K)$ need not have $|UFE(K)|<\mu$---so the argument in Theorem~\ref{niceacckaec} cannot simply be repeated here.  Still, $U$ can only do so much damage:

\begin{lemma} The functor $U:\cl_2\to\Set$ sends $\mu$-presentable objects to $\nu^+$-presentable objects, where $\nu=\scard{\Mor(\ca)}$.\end{lemma}

\begin{theorem}\label{equivthm}Let $\ck$ be a $\mu$-accessible category with all morphisms mono.  Then $\ck$ is equivalent to a $\mu$-AEC with 
L\"owenheim-Skolem-Tarski number $\lambda=\max(\mu,\nu)^{<\mu}$.\end{theorem}
\begin{proof} Consider $X\subseteq FE(K)$, and let $\alpha=\scard{X}$.  Since $\cl_2$ is $\mu$-accessible, it is $(\alpha^{<\mu})^+$-accessible (provided $\alpha\geq\mu$); see the proof of \ref{niceacckaec}.  Thus there is an $(\alpha^{<\mu})^+$-presentable $\cl_2$-subobject $M_X$ of $FE(K)$ with $X\subseteq M_X$.  By Theorem 2.3.11 in \cite{makkaiparebook}, $M_X$ can be expressed as an $(\alpha^{<\mu})^+$-small $\mu$-directed colimit of $\mu$-presentables in $\cl_2$, meaning that $U(M_X)$ is an $(\alpha^{<\mu})^+$-small $\mu$-directed colimit of sets of size less or equal than $\nu$.  This is of cardinality less or equal than $\alpha^{<\mu}+\max(\mu,\nu)$.  This suggests $\max(\mu,\nu)$ might serve as our L\"owenheim-Skolem-Tarski number, but we must fulfill the requirement that $\lambda^{<\mu}=\mu$.  So, take $\lambda=\max(\mu,\nu)^{<\mu}$.\end{proof}

The $\mu$-AEC from \ref{equivthm} is a full subcategory of $\Emb(L)$ where $L$ is a finitary language.  Although this equivalence destroys both the ambient language and the underlying sets, and thus moves beyond the methods usually entertained in model theory, it allows us to transfer intuition and concepts between the two contexts.  

The equivalence allows us to generate the notion of a L\"{o}wenheim-Skolem number in an accessible concrete category, where concreteness is necessary to form the question.

\begin{prop}
Let $(\K, U)$ be a $\mu$-accessible concrete category with all maps monomorphisms such that $U$ preserves $\mu$-directed colimits.  Then if $M \in \K$ and $X \subset UM$, there is a subobject $M_0 \in \K$ of $M$ such that $X \subset UM_0$ and $M_0$ is $(\scard{X} ^{<\mu})^+$-presentable.
\end{prop}

Note that we have proved that every $\mu$-accessible category with all maps monomorphism has such a concrete functor:  by Theorem \ref{equivthm}, there is a (full and faithful) equivalence $F: \K \to \K'$, where $\K'$ is some $\mu$-AEC.  The universe functor $U:\K' \to \Set$ if faithful and preserves $\mu$-directed colimits, so $FU: \K \to \K'$ does as well.

\begin{proof} 

Let $M \in \K$, where $\K$ is $\mu$-accessible and concrete with monomorphisms.  Let $X \subset UM$.  We want to find $M_0 \leq M$ with $X \subset UM_0$ that is $(\scard{X}^{<\mu})^+$-presentable.  By accessibility, we can write $M$ as a $\mu$-directed colimit $\seq{M^i, f_{j, i} \mid j < i \in I}$ where $M^i \in \K$ is $\mu$-presentable, $I$ is $\mu$-directed, and $f_{i, \infty}$ are the colimit maps.  

Because $U$ preserves $\mu$-directed colimits, there is $I_0 \subset I$ of size $\leq \scard{X}$ such that, for every $x \in X$, there is some $i_x \in I_0$ such that $x \in Uf_{i_x,\infty}M^i$. Close this to a $\mu$-directed subset $I_1 \subset I$ of size $\leq \scard{X}^{<\mu}$ and let $(M^*, f_{i, *})$ be the colimit of $\{M^i, f_{j, i} \mid j<i \in I_1\}$.  Since this system also embeds into $M$, there is a canonical map $f^*:M^* \to M$.  Set $M_0 = f^*M^*$.  Then $M_0$ is a subobject of $M$ and $X \subset UM_0$, so we just need to show $M_0$ is $(\scard{X}^{<\mu})^+$-presentable.  This follows from \cite{adamekrosicky} 1.16: since $\mu \leq (\scard{X}^{<\mu})^+$, each $M^i$ is $(\scard{X}^{<\mu})^+$ presentable.  Since $\scard{I_1} \leq \scard{X}^{<\mu}$, $M_0$ is $(\scard{X}^{<\mu})^+$-presentable by the cited result.


\end{proof}

Although the previous theorem doesn't use any model theoretic properties directly, it is inspired by standard proofs of the downward L\"{o}wenheim-Skolem theorem and seems not to have been known previously.

Going the other direction, knowledge about accessible categories allows us to show that $\mu$-AECs do not, in general, admit Ehrenfeucht-Mostowski constructions.  In particular, not every $\mu$-AEC $\ck$ admits a faithful functor $E:\Lin\to\ck$:

\begin{exam}\label{no-em-model}

\em{Let $\ck$ be the category of well-ordered sets and order-preserving injections.  By \cite{adamekrosicky} 2.3(8), $\ck$ is $\omega_1$-accessible, and clearly all of its morphisms are monomorphisms.  By Theorem~\ref{equivthm}, it is therefore equivalent to an $\omega_1$-AEC.  As $\ck$ is isomorphism rigid---that is, it contains no nonidentity isomorphisms---it cannot admit a faithful functor from $\Lin$, which is far from isomorphism rigid.}\end{exam}

Ehrenfeucht-Mostowski constructions are a very powerful tool in the study of AECs (see for example \cite{sh394}). This suggests that $\mu$-AECs may be too general to support a robust classification theory.  In particular, the lack Ehrenfeucht-Mostowski models, in turn, means that there is no analogue of the Hanf number that has proven to be very useful in the study of AECs.  

A possible substitute to the notion of Hanf number is that of \emph{LS-accessibility}, which was introduced by \cite{bekerosicky}.  Rather than looking at the cardinality of the models, they asked about the internal size, as computed in the category.  The shift stems from the following: it is clear that there are $\aleph_1$-AECs that don't have models in arbitrarily large cardinalities: looking at complete (non-discrete) metric spaces or \cite[Example 4.8]{bekerosicky}, there can be no models in cardinalities satisfying $\lambda < \lambda^\omega$.  However, the internal size based on presentability rank mentioned above gives that, e. g., complete metric spaces have models of all \emph{sizes}.  Thus, an accessible category is called LS-accessible iff there is a threshold such that there are object of every \emph{size} above that threshold.  Beke and Rosicky \cite{bekerosicky} ask if every large accessible category is LS-accessible.  This question is still open and a positive answer (even restricting to accessible categories where all maps are mono) would aid the analysis of $\mu$-AECs (see the discussion at the start of Section \ref{classthysect}).  

Still, under the additional assumption of upper bounds for increasing chains of structures---\emph{directed bounds}, in the language of \cite{rosickyjsl}, or the \emph{$\delta$-chain extension property}, defined below---we can rule out Example~\ref{no-em-model}, and begin to develop a genuine classification theory.

\section{Tameness and large cardinals}\label{tamenesssect}

In \cite{tamelc-jsl}, it was shown by the first author that, assuming the existence of large cardinals, every AEC satisfies the important locality property know as tameness.  Tameness was isolated (from an argument of Shelah \cite{sh394}) by Grossberg and VanDieren in \cite{tamenessone}, and was used to prove an upward categoricity transfer from a successor cardinal in \cite{tamenesstwo, tamenessthree}. Tame AECs have since been a very productive area of study. For example, they admit a well-behaved notion of independence \cite{ss-tame-toappear-v3, indep-aec-v5} and many definitions of superstability can be shown to be equivalent in the tame context \cite{gv-superstability-v2}. 

In this section, we generalize Boney's theorem to $\mu$-AECs (in a sense, this also partially generalizes the recent \cite{boney-zambrano-tamelc-v1} which proved an analogous result for metric AECs, but for a stronger, metric specialization of tameness). We start by recalling the definition of tameness (and its generalization: full tameness and shortness) to this context. This generalization already appears in \cite[Definition 2.21]{sv-infinitary-stability-v5}.

\begin{defin}[Definitions 3.1 and 3.3 in \cite{tamelc-jsl}]\label{shortness-def}
  Let $\K$ be an abstract class and let $\kappa$ be an infinite cardinal.

  \begin{enumerate}
    \item $\K$ is \emph{$(<\kappa)$-tame} if for any distinct $p, q \in \gS (M)$, there exists $A \subseteq \scard{M}$ such that $|A| < \kappa$ and\footnote{We use here Galois types over sets as defined in \cite[Definition 2.16]{sv-infinitary-stability-v5}.}  $p \rest A \neq q \rest A$.
    \item $\K$ is \emph{fully $(<\kappa)$-tame and short} if for any distinct $p, q \in \gS^{\alpha} (M)$, there exists $I \subseteq \alpha$ and $A \subseteq |M|$ such that $\scard{I} + \scard{A} < \kappa$ and $p^I \rest A \neq q^I \rest A$.
    \item We say $\K$ is \emph{tame} if it is $(<\kappa)$-tame for some $\kappa$, similarly for fully tame and short.
  \end{enumerate}
\end{defin}

Instead of strongly compact cardinals, we will (as in \cite{lc-tame} and \cite{btr}) use \emph{almost} strongly compact cardinals:

\begin{defin}
  An uncountable limit cardinal $\kappa$ is \emph{almost strongly compact} if for every $\mu < \kappa$, every $\kappa$-complete filter extends to a $\mu$-complete ultrafilter.
\end{defin}

Note that the outline here follows the original model theoretic arguments of \cite{tamelc-jsl}.  The category theoretic arguments of \cite{LRclass} and \cite{btr} can also be used.

A minor variation of the proof of \L o\'{s}'s theorem for $L_{\kappa, \kappa}$ (see \cite[Theorem 3.3.1]{Dic}) gives:

\begin{fact}\label{ultraprod-fact}
  Let $\kappa$ be an almost strongly compact cardinal. Let $\mu < \kappa$, let $(M_i)_{i \in I}$ be $L$-structures, and let $U$ be a $\mu^+$-complete ultrafilter on $I$. Then for any formula $\phi \in L_{\mu, \mu}$, $\prod M_i\backslash U \models \phi[[f]_U]$ if and only if $M_i \models \phi[f (i)]$ 
for $U$-almost all $i \in I$.
\end{fact}

Using the presentation theorem, we obtain \L o\'{s}'s theorem for $\mu$-AECs:

\begin{lemma}
  Let $\K$ be a $\mu$-AEC. Let $(M_i)_{i \in I}$ be models in $\K$ and let $U$ be a $\left(2^{\LS (\K)}\right)^+$-complete ultrafilter on $I$. Then 
  $\prod M_i\backslash U \in \K$.
\end{lemma}
\begin{proof}[Proof sketch]
  Let $\mu := \left(2^{\LS (\K)}\right)^+$. By the presentation theorem (Theorem \ref{pres-thm}), there exists a language $L' \supseteq L (\K)$ and a sentence $\phi \in L_{\mu, \mu}'$ such that $\K = \Mod (\phi) \rest L = \K$. Now use Fact \ref{ultraprod-fact} together with the proof of \cite[Theorem 4.3]{tamelc-jsl}.
\end{proof}

All the moreover clauses of \cite[Theorem 4.3]{tamelc-jsl} are also obtained, thus by the same proof as \cite[Theorem 4.5]{tamelc-jsl}, we get:

\begin{theorem}\label{tamelc-muaecs}
  Let $\K$ be a $\mu$-AEC and let $\kappa > \LS (\K)$ be almost strongly compact. Then $\K$ is fully $(<\kappa)$-tame and short.
\end{theorem}

In particular, if there is a proper class of almost strongly compact cardinals, every $\mu$-AEC is fully tame and short.  Using the recent converse for the special case of AECs due to Boney and Unger \cite{lc-tame}, we obtain also a converse in $\mu$-AECs:

\begin{theorem}
  The following are equivalent:

  \begin{enumerate}
    \item\label{equiv-1} For every $\mu$, every $\mu$-AEC is fully tame and short.
    \item\label{equiv-2} Every AEC is tame.
    \item\label{equiv-3} There exists a proper class of almost strongly compact cardinals.
  \end{enumerate}
\end{theorem}
\begin{proof}
  (\ref{equiv-1}) implies (\ref{equiv-2}) is because AECs are $\aleph_0$-AECs. (\ref{equiv-2}) implies (\ref{equiv-3}) is \cite{lc-tame} and (\ref{equiv-3}) implies (\ref{equiv-1}) is Theorem \ref{tamelc-muaecs}.
\end{proof}

\section{On categorical $\mu$-AECs}\label{classthysect}

Here we show that some non-trivial theorems of classification theory for AECs transfer to $\mu$-AECs and, by extension, accessible categories with monomorphisms. Most of the classification theory for AECs has been driven by Shelah's categoricity conjecture\footnote{For more references and history, see the introduction of Shelah's book \cite{shelahaecbook}}. For an abstract class $\K$, we write $I(\lambda,\K)$ for the number of pairwise non-isomorphic models of $\K$ of cardinality $\lambda$. An abstract class $\K$ is said to be \emph{categorical in $\lambda$} if $I(\lambda,\K) = 1$. Inspired by Morley's categoricity theorem, Shelah conjectured:

\begin{conj}\label{shconj}
  If an AEC is categorical in a high-enough cardinal, then it is categorical on a tail of cardinals.
\end{conj} 

Naturally, one can ask the same question for both $\mu$-AECs and accessible categories, where, following \cite{rosickyjsl}, we say an accessible category is \emph{categorical in $\lambda$} if it contains exactly one object of internal size $\lambda$ (up to isomorphism).  By shifting the question to these more general frameworks, of course, we make it more difficult to arrive at a positive answer.  If the answer is negative, on the other hand, counterexamples should be more readily available in our contexts: if indeed the answer is negative, this would give us a bound on the level of generality at which the categoricity conjecture can hold.

\begin{question} \label{pbacc}
If a large accessible category (whose morphisms are monomorphisms) is categorical in a high-enough cardinal, is it categorical
on a tail of cardinals?
\end{question}

 A negative answer to the question of Beke and Rosicky from Section \ref{access-sec}---an example of an large accessible category $\ck$ with arbitrarily large gaps in internal sizes---would yield a negative answer to Question \ref{pbacc}: as noted in \cite{bekerosicky} 6.3, it suffices to take the coproduct $\ck\coprod\Set$. This adds exactly one isomorphism class to each size, resulting in a category that is (internally) categorical in arbitrarily high cardinals---the gaps of $\ck$---but also fails to be (internally) categorical in arbitrarily large cardinals. By taking injective mappings of sets, one can do the same for large accessible categories whose morphisms are monomorphisms. \cite{bekerosicky} and \cite{LRclass} contain sufficient conditions for LS-accessibility: in particular, it is enough to add the assumption of the existence of arbitrary directed colimits (see \cite{LRclass}, 2.7).

For $\mu$-AECs, the natural formulation is in terms not of the internal size, but of the cardinality of underlying sets.  Some adjustments have to be made, as a $\mu$-AEC need not have a model of cardinality $\lambda$ when $\lambda^{<\mu} > \lambda$, and thus eventual categoricity would fail more or less trivially.   

\begin{question}\label{pbaec}
  If a $\mu$-AEC is categorical in a high-enough cardinal $\lambda$ with $\lambda = \lambda^{<\mu}$, is it categorical in all sufficiently high $\lambda'$ such that $\lambda' = (\lambda')^{<\mu}$.
\end{question} 

For $\mu=\omega$, this question reduces to \ref{shconj}.

\begin{remark}
We will show that a positive answer to Question~\ref{pbacc}, the internal version, implies, at the very least, a positive answer to Question~\ref{shconj}. Let $\ck$ be an AEC in a language $L$. 
Then $\ck$ is an accessible category and, following \cite{bekerosicky} 4.1, 4.3 and 3.6, there is a regular cardinal $\kappa$ such that $\ck$
is $\kappa$-accessible and $E$ preserves sizes $\lambda\geq\kappa$. We can assume that, in $\Emb(L)$, they coincide with cardinalities of underlying sets.
Thus, any $K_1,K_2$ with sufficiently large and distinct $\mcard{EK_1},\mcard{EK_2}$ have distinct sizes $\mcard{K_1},\mcard{K_2}$  and thus $K_1$ and $K_2$ are not isomorphic.

At present we do not know whether a positive answer to \ref{pbacc} implies a positive answer to \ref{pbaec}.
\end{remark}

Of course \ref{pbacc} is currently out of reach, as is \ref{pbaec}. We are not sure about the truth value of either one: it is plausible that there are counterexamples. A possible starting point for \ref{pbaec} would be to use Theorem \ref{tamelc-muaecs} to try to generalize \cite{tamelc-jsl} to $\mu$-AECs categorical in an appropriate successor above a strongly compact (see also \cite{sh1019}, which proves some model-theoretic results for classes of models of $L_{\kappa, \kappa}$ with $\kappa$ a strongly compact cardinal).

We show here that some facts which follow from categoricity in AECs also follow from categoricity in $\mu$-AECs. As in \cite{rosickyjsl}, which considers categoricity in accessible categories with directed bounds (and, ultimately, directed colimits), we have to add the following hypothesis:

\begin{defin}
  Let $\delta$ be an ordinal. An abstract class $\K$ has the \emph{$\delta$-chain extension property} if for every chain $\seq{M_i : i < \delta}$, there exists $M_\delta \in \K$ such that $M_i \lea M_\delta$ for all $i < \delta$. We say that $\K$ has the \emph{chain extension property} if it has the $\delta$-chain extension property for every limit ordinal $\delta$.
\end{defin}
\begin{remark}
  If $\K$ is a $\mu$-AEC, then $\K$ has the chain extension property if and only if $\K$ has the $\delta$-chain extension property for every limit $\delta < \mu$.
\end{remark}
\begin{remark}
  $\mu$-CAECs have the chain extension property (recall the item (\ref{caec-ex}) from the list of examples). Moreover, any $\mu$-AEC naturally derived from\footnote{This can be made precise using the notion of a \emph{skeleton}, see \cite[Definition 5.3]{indep-aec-v5}.} an AEC (such as the class of $\mu$-saturated models of an AEC) will have the chain extension property.
\end{remark}

We adapt Shelah's \cite[Theorem IV.1.12.(1)]{shelahaecbook} to $\mu$-AECs:

\begin{theorem}\label{inflog-thm}
  Let $\K$ be a $\mu$-AEC. Let $\lambda \ge \LS (\K)$ be such that $\lambda = \lambda^{<\mu}$ and $\K_\lambda$ has the $\delta$-chain extension property for all limit $\delta < \lambda^+$. Assume $\K$ is categorical in $\lambda$. Let $M, N \in \K_{\ge \lambda}$. If $M \lea N$, then $M \lee_{L_{\infty, \mu}} N$. 
\end{theorem}

Notice that the cardinal arithmetic ($\lambda^{<\mu} = \lambda$) is a crucial simplifying assumption in the AEC version that Shelah later worked to remove (see \cite[Section IV.2]{shelahaecbook} and \cite{bvcatinflog}).  It appears naturally here in the context of a $\mu$-AEC, but note that the chain extension might guarantee the existence of models of intermediate sizes (i.e.\ in $\chi < \chi^{<\mu}$).  



\begin{proof}[Proof of Theorem \ref{inflog-thm}]
  We first assume that $M, N \in \K_\lambda$ and $M \lea N$. Let $\phi (\by)$ be an $L_{\infty, \mu}$-formula with $\ell (\by) = \alpha < \mu$ and let $\ba \in \fct{\alpha}{|M|}$. We show that $M \models \phi[\ba]$ if and only if $N \models \phi[\ba]$ by induction on the complexity of $\phi$. If $\phi$ is atomic, this holds because $M \subseteq N$. If $\phi$ is a boolean combination of formulas of lower complexity, this is easy to check too. So assume that $\phi (\by) = \exists \bx \psi (\bx, \by)$. If $M \models \phi[\ba]$, then using induction we directly get that $N \models \phi[\ba]$. Now assume $N \models \phi[\ba]$, and let $\bb \in \fct{<\mu}{|N|}$ be such that $N \models \psi[\bb, \ba]$. 

  We build an increasing chain $\seq{M_i : i < \lambda^+}$ and $\seq{f_i, g_i : i < \lambda^+}$ such that for all $i < \lambda^+$:

  \begin{enumerate}
    \item $M_i \in \K_\lambda$
    \item If $\cf{i} \ge \mu$, then $M_i = \bigcup_{j < i} M_j$.
    \item $f_i : M \cong M_i$.
    \item $g_i : N \cong M_{i + 1}$.
    \item $f_i \subseteq g_i$.
  \end{enumerate}

  \paragraph{\underline{This is possible}}
  If $i = 0$, let $M_0 := M$. For any $i$, given $M_i$, use categoricity to pick $f_i : M \cong M_i$ and extend it to $g_i : N \cong M_{i + 1}$. If $i$ is limit and $\cf{i} \ge \mu$, take unions. If $\cf{i} < \mu$, use the chain extension property to find $M_i \in \K_\lambda$ such that $M_j \lea M_i'$ for all $j < i$. 
  \paragraph{\underline{This is enough}}
  For each $i < \lambda^+$, let $\alpha (i)$ be the least $\alpha < \lambda^+$ such that $\ran(f_i (\ba)) \subseteq |M_{\alpha}|$. Let $S := \{i < \lambda^+ \mid \cf{i} \ge \mu\}$. Note that $S$ is a stationary subset of $\lambda^+$ and the map $i \mapsto \alpha (i)$ is regressive on $S$. By Fodor's lemma, there exists a stationary $S_0 \subseteq S$ and $\alpha_0 < \lambda^+$ such that for any $i \in S_0$, $\alpha (i) = \alpha_0$, i.e.\ $\ran (f (\ba_i)) \subseteq |M_{\alpha_0}|$. Now $\scard{\fct{<\mu}{|M_{\alpha_0}|}} = \lambda^{<\mu} = \lambda$ and $|S_0| = \lambda^+$ so by the pigeonhole principle there exists $i < j$ in $S_0$ such that $f_i (\ba) = f_j (\ba)$. Now, since $N \models \psi[\bb, \ba]$, we must have $M_{i + 1} \models \psi[g_i (\bb), g_i (\ba)]$. By the induction hypothesis, $M_j \models \psi[g_i (\bb), g_i (\ba)]$. Thus $M_j \models \phi[g_i (\ba)]$. Since $f_i \subseteq g_i$, $g_i (\ba) = f_i (\ba)$ so $M_j \models \phi[f_i (\ba)]$. Since $f_i (\ba) = f_j (\ba)$, we have that $M_j \models \phi[f_j (\ba)]$. Applying $f_j^{-1}$ to this equation, we obtain $M \models \phi[\ba]$, as desired.

  This proves the result in case $M, N \in \K_\lambda$. If $M, N \in \K_{\ge \lambda}$ and $M \lea N$, then, as before, we can find a $\mu$-directed system $\seq{N_\ba \in \K_\lambda : \ba \in {}^{<\mu} N}$ with colimit $N$ such that $\ba \in |N_\ba|$ and, if $\ba \in {}^{<\mu} M$, then $N_\ba \lea M$.
  
  As before we prove by induction on $\phi \in L_{\infty,\mu}$ that $M \models \phi[\ba]$ if and only if $N \models \phi[\ba]$. The interesting case is when $\phi = \exists \bx \psi (\bx, \by)$ and the left to right direction is straightforward, so assume $N \models \phi[\ba]$, i. e.,\ there exists $\bb \in \fct{<\mu}{|N|}$ such that $N \models \psi[\bb, \ba]$. By the previous part, $N_\ba \lee_{L_{\infty, \mu}} N_{\ba\bb}$.  So there is $\bb' \in N_\ba$ such that $N_\ba \vDash \psi[\bb', \ba]$.  Since $N_\ba \lea M$, by induction, we have $M \vDash \phi[\ba]$.
\end{proof}

Another result that can be adapted is Shelah's famous combinatorial argument that amalgamation follows from categoricity in two successive cardinals \cite[Theorem 3.5]{Sh88}. We start with some simple definitions and lemmas:

\begin{defin}
  Let $\mu \le \lambda$ be regular cardinals. $C \subseteq \lambda$ is a \emph{$\mu$-club} if it is unbounded and whenever $\seq{\alpha_i : i < \delta}$ is increasing in $C$ with $\mu \le \cf{\delta} < \lambda$, then $\sup_{i < \delta} \alpha_i \in C$.
\end{defin}
\begin{remark}
  So $\aleph_0$-club is the usual notion of club.
\end{remark}

\begin{lemma}\label{club-reflection}
  Let $\mu$ be a regular cardinal. Assume $\K$ is a $\mu$-AEC and $\lambda \ge \LS (\K)$. Let $\seq{M_i^\ell : i < \lambda^+}$, $\ell = 1,2$, be increasing in $\K_\lambda$ such that for all $i < \lambda^+$ with $\cf{i} \ge \mu$, $M_i^\ell = \bigcup_{j < i} M_j^\ell$. 

  If $f: \bigcup_{i < \lambda^+} M_i^1 \cong \bigcup_{i < \lambda^+} M_i^2$, then the set $\{i < \lambda^+ \mid f \rest M_i^1: M_i^1 \cong M_i^2\}$ is a $\mu$-club.
\end{lemma}
\begin{proof}
  Let $C := \{i < \lambda^+ \mid f \rest M_i^1: M_i^1 \cong M_i^2\}$. By cardinality considerations, for each $i < \lambda^+$, there is $j_i < \lambda^+$ such that $|f(M_i^1)| \subseteq |M_{j_i}^2|$ (by coherence this implies $f(M_i^1) \lea M_{j_i}^2$). Let $\delta$ have cofinality $\mu$ such that for all $i < \delta$, $j_i < \delta$. Then by continuity $f(M_\delta^1) \lea M_{\delta}^2$. Let $C_0$ be the set of all such $\delta$. It is easy to check that $C_0$ is a $\mu$-club. Similarly, let $C_1$ be the set of all $\delta$ such that $f^{-1} (M_\delta^2) \lea M_\delta^1$. $C_1$ is also a $\mu$-club and it is easy to check that $C = C_0 \cap C_1$, and the intersection of two $\mu$-clubs is a $\mu$-club, so the result follows.
\end{proof}

\begin{theorem}\label{ap-from-categ}
  Let $\mu$ be a regular cardinal. Assume $\K$ is a $\mu$-AEC, $\lambda = \lambda^{<\mu} \ge \LS (\K)$, $I(\lambda,\K) = 1\leq I (\lambda^+,\K) < 2^{\lambda^+}$. If:

  \begin{enumerate}
    \item $\K_\lambda$ has the extension property for $\delta$-chains (see above) for every $\delta < \lambda^+$.
    \item $\lambda = \lambda^\mu$ and $2^{\lambda} = \lambda^+$.

  \end{enumerate}

  Then $\K$ has $\lambda$-amalgamation.
\end{theorem} 
\begin{proof}
  Assume not. By failure of amalgamation and some renaming, we have:

  \begin{itemize}
    \item[$(\ast)$] If $M_1, M_2 \in \K_\lambda$ and $f: M_1 \cong M_2$, there are $M_l' \in \K_\lambda$, $\ell = 1,2$, with $M_l \lea M_l'$, $\scard{|M_l'| - |M_l|} = \lambda$, such that there is no $N \in \K_\lambda$ and $g_l: M_l' \rightarrow N$ commuting with $f$.
  \end{itemize}

  In particular (taking $M_1 = M_2$ and $f$ the identity function), the model of size $\lambda$ is not maximal. By Gregory's theorem (see \cite[Theorem 23.2]{jechbook}), the combinatorial principle $\Diamond_{E_\mu}$ holds, where $E_\mu := \{i < \lambda^+ \mid \cf{i} \ge \mu\}$. With some coding, one can see that $\Diamond_{E_\mu}$ is equivalent to: 

  \begin{itemize}
    \item[$(\ast \ast)$] There are $\{\eta_\alpha, \nu_\alpha: \alpha \rightarrow \mid \alpha < \lambda^+\}, \{g_\alpha: \alpha \rightarrow \alpha \mid \alpha < \lambda^+\}$ such that for all $\eta, \nu: \lambda^+ \rightarrow 2$, $g: \lambda^+ \rightarrow \lambda^+$, the set $\{\alpha \in E_\mu \mid \eta_\alpha = \eta \upharpoonright \alpha, \nu_\alpha = \nu \upharpoonright \alpha, g_\alpha = g \restriction \alpha\}$ is stationary.
  \end{itemize}

  We build a strictly increasing tree $\{M_\eta \mid \eta \in \fct{\le \lambda^+}{2}\}$ such that:

\begin{enumerate}
  \item\label{prop-1} $|M_\eta| \subseteq \lambda^+$, $\mcard{M_\eta} = \lambda$, $\ell (\eta) \in |M_{\eta \smallfrown \ell}|$ for all $\eta \in \fct{<\lambda^+}{2}$ and $\ell < 2$.
  \item If $\eta \in \fct{\le\lambda^+}{2}$ and $\cf{\ell (\eta)} \ge \mu$, then $M_\eta = \bigcup_{j < \ell (\eta)} M_{\eta \rest j}$.
  \item\label{prop-3} If $|M_{\eta_\delta}| = \delta$, $\eta_\delta \neq \nu_\delta$, and $g_\delta : M_{\eta_\delta} \cong M_{\nu_\delta}$ is an isomorphism, for any $\ell, \ell' < 2$ and any $\nu \supseteq \nu_\delta \smallfrown \ell'$, $g_\delta$ \emph{cannot} be extended to an embedding of $M_{\eta_\delta \smallfrown \ell}$ into $M_{\nu}$.
\end{enumerate}

\paragraph{\underline{This is enough}}

We claim that for any $\eta \neq \nu \in \fct{\lambda^+}{2}$, $M_\eta \not \cong M_\nu$. Indeed, assume $f: M_\eta \rightarrow M_\nu$ is an isomorphism. For $i < \lambda^+$, let $f_i := f \rest M_{\eta \rest i}$ and let $C := \{i < \lambda^+ \mid f_i : M_{\eta \rest i} \cong M_{\nu \rest i}\}$. By Lemma \ref{club-reflection}, $C$ is a $\mu$-club. Also $\{i < \lambda^+ \mid |M_{\eta \rest i}| = i\}$ is a club so without loss of generality is contained in $C$. Now the stationary set described by $(\ast \ast)$ intersects $C$ in unboundedly many places (as it only has points of cofinality $\mu$), hence there is $\delta < \lambda^+$ such that $\eta \upharpoonright \delta \neq \nu \upharpoonright \delta$, $\eta_\delta = \eta \upharpoonright \delta, \nu_\delta = \nu \upharpoonright \delta$, $g_\delta = f \upharpoonright \delta$, $\delta = |M_{\eta_\delta}| = |M_{\nu_\delta}|$, and $g_\delta: M_{\eta_\delta} \cong M_{\nu_\delta}$. But $f$ extends $g_\delta$ and restricts to an embedding of $M_{\eta \smallfrown \eta (\delta)}$ into $M_{\nu \restriction \gamma}$, for some $\gamma < \lambda^+$ with $\gamma > \delta$ sufficiently large. This contradicts (\ref{prop-3}).

\paragraph{\underline{This is possible}}

Take any $M_{<>} \in \K$ with $|M_{\seq{}}| = \lambda$ for the base case, take unions at limits of cofinality at least $\mu$, and use the extension property for chains (and some renaming) at limits of cofinality less than $\mu$.

Now if one wants to define $M_{\eta \smallfrown l}$ for $\eta \in \fct{\delta}{2}$ (assuming by induction that $M_\nu$ for all $\nu \in \fct{\le\delta}{2}$ have been defined) take any two strict extensions, unless $|M_\eta| = \delta$, $\eta_\delta \neq \nu_\delta$, $g_\delta: M_{\eta_\delta} \cong M_{\nu_\delta}$ is an isomorphism, and either $\eta = \eta_\delta$, or $\eta = \nu_\delta$. We show what to do when $\eta = \eta_\delta$. The other case is symmetric. Let $M_{\eta_\delta}'$, $M_{\nu_\delta}'$ be as described by $(\ast)$ and let $M_{\eta_\delta \smallfrown l}, M_{\nu_\delta \smallfrown l}$ be their appropriate renaming to satisfy (\ref{prop-1}). Now $(\ast)$ tells us that (\ref{prop-3}) is satisfied.
\end{proof}
\begin{remark}
  Of course, the set-theoretic hypotheses of Theorem \ref{ap-from-categ} can be weakened. For example, it is enough to require $\lambda = \lambda^{<\mu}$ and $\Diamond_{S_\mu}$ or even (by Shelah's more complicated proof) a suitable instance of the weak diamond.  It is not clear, however, that it follows from just $2^{\lambda} < 2^{\lambda^+}$.
\end{remark}

\end{document}